\documentclass[12pt]{article}
\usepackage[left=3cm,top=3.0cm,right=3cm,bottom=2.6cm]{geometry}

\usepackage[ansinew]{inputenc}
\usepackage{amscd,amsfonts,amsmath,amssymb,amsthm,enumerate,epsfig,float,graphics,graphicx,pst-all,yhmath}
\newcommand{\inte}{\operatorname*{int}}
\newcommand{\aff}{\operatorname*{aff}}

\newcommand{\Rd}{\mathbb{R}^2}
\newcommand{\Rt}{\mathbb{R}^3}
\newcommand{\Rn}{\mathbb{R}^{n}}

\newtheorem{lemma}{Lemma}
\newtheorem{proposition}{Proposition}
\newtheorem{theorem}{Theorem}
\newtheorem{corollary}{Corollary}

\newtheorem{remark}{Remark}
\newtheorem{definition}{Definition}
\newtheorem{conjecture}{Conjecture}
\newtheorem{question}{Question}

\title {A characterization of the sphere and a body of revolution by means of Larman points}

\author{M. Angeles Alfonseca, M. Cordier, J. Jer\'onimo-Castro$^{3}$ \\ E. Morales-Amaya$^{4}\footnote{{This research was supported by the National Council of Sciences and Technology of Mexico (CONACyT) Grant I0110/62/10 and SNI 21120. This work was partially done during a sabbatical year of the author to the Depart. of Mathematics, University College London (UCL). The author wants to thanks the hospitality and the support given by the UCL to our scientific work.}}$\\ 
\small{$^{3}$Facultad de Ingenier\'ia}\\
\small{Universidad Aut\'onoma de Quer\'etaro, M\'exico}\\
\small{$^{4}$Facultad de Matem\'aticas-Acapulco,}\\
\small{Universidad Aut\'onoma de Guerrero, M\'exico}\\
 }

\begin{document}

\maketitle
	
\abstract{Let $K\subset \Rn$, $n\geq 3$, be a convex body.  A point $p\in \inte K$ is said to be a \textit{Larman point} of $K$ if  for every hyperplane $\Pi$ passing through $p$ 
the section $\Pi\cap K$ has a $(n-2)$-plane of symmetry. If $p$ is a Larman point of $K$ and, in addition, for every section $\Pi\cap K$, $p$ is in the corresponding $(n-2)$-plane of symmetry, then we call $p$ a {\it revolution} point of $K$. We conjecture that if $K$ contains a Larman point which is not a revolution point, then $K$ is either an ellipsoid or a body of revolution. This generalizes a conjecture of K. Bezdek for convex bodies in $\mathbb{R}^3$ to $n \geq 4$. 
We prove several results related to the conjecture for strictly convex origin symmetric bodies. Namely, if $K \subset \mathbb{R}^n$ is a strictly convex origin symmetric body that contains a revolution point $p$ which is not the origin, then $K$ is a body of revolution. This generalizes the False Axis of Revolution Theorem proven in \cite{jemomo}. We also show that if $p$ is a Larman point of $K \subset \mathbb{R}^3$ and there exists a line $L$ such that $p\notin L$ and, for every plane $\Pi$ passing through $p$, the line of symmetry of the section $\Pi \cap K$ intersects  $L$, then  $K$ is a body of revolution (in some cases, we conclude that $K$ is a sphere). We obtain a similar result for projections of $K$. Additionally, for $K \subset \mathbb{R}^n$, $n \geq 4$, we show that if every hyperplane section or projection of $K$ is a body of revolution and $K$ has a unique diameter $D$, then $K$ is a body of revolution with axis $D$.}

\section{Introduction}

In this work, we study Geometric Tomography problems in which we are given information about the symmetries of sections (or projections) of a convex body $K \subset \Rn$, $n\geq 3$, and want to obtain 
information about the symmetries of $K$.

\begin{question}\label{prob3} 
What can we say about a convex body $K\subset \Rn$, $n \geq 3$, with the property that there exists a point $p\in \Rn$ such that all hyperplane sections of $K$ passing through $p$ possess a certain type of symmetry? 

\end{question}

A particularly simple case of Question \ref{prob3} occurs when  $p$ is an interior point of $K$ and all the hyperplane sections passing through $p$ are discs. In this case, $K$ must be a sphere.   Indeed, by Hammer's result \cite[Thm. 3.1]{ham}, there is a diametral chord of $K$ passing through $p$. Since for every hyperplane $\Pi$ through $p$, the section $\Pi\cap K$ is a $(n-1)$-dimensional ball, 
the diametral chord is an axis of revolution of $K$. Therefore, $K$ is a solid sphere.

Another case of Question \ref{prob3}  occurs when all the sections of $K$ through the point $p$ are assumed to be centrally symmetric, but $p$ is not the center of symmetry of $K$. This problem is known as the False Centre Theorem  of Aitchison-Petty-Rogers and Larman. In \cite{ro1}, Rogers proved, in a very elegant way, that  if $K\subset \Rn$, $n\geq 3$, has a false centre, then $K$ must be centrally symmetric. In the same paper, Rogers conjectured that such a $K$ should be an ellipsoid.   The conjecture was confirmed in \cite{apr} in the case when the false centre is an interior point of $K$. Finally,  the False Centre Theorem was proven in all its generality in \cite{la}.

Instead of considering sections with central symmetry, as in the False Centre Theorem, K. Bezdek formulated the following conjecture in which the sections have axial symmetry.

\begin{conjecture}\cite[pg. 221]{blnp}\label{bez}
    If all plane sections of a convex body $K \subset \mathbb{R}^3$ have an axis of symmetry, then $K$ is an ellipsoid or a body of revolution.
\end{conjecture}

Conjecture \ref{bez} involves all sections of $K$, not just the sections passing through a fixed point. In \cite{mo}, Montejano gave an example showing that considering only sections through a fixed point is not enough. Indeed, in $\mathbb{R}^3$, the convex hull of two perpendicular discs centered at the origin has the property that every section through the origin has an axis of symmetry. 

K. Bezdek's conjecture can be generalized to higher dimensions in several ways. Let $K \subset \Rn$, $n \geq 4$, be a convex body, and consider all hyperplane sections of $K$. We may assume one of the following hypotheses:

\begin{enumerate}[H(i):]
\item All hyperplane sections have an axis of symmetry,
   
    \item All hyperplane sections have an $(n-2)$-plane of symmetry,
     \item All hyperplane sections are $(n-1)$-dimensional bodies of revolution.
\end{enumerate}

When $n=3$, all three conditions H(i), H(ii) and H(iii) reduce to K. Bezdek's condition that all plane sections of $K$ have an axis of symmetry. 

For $n\geq 4$, in cases H(i) and H(ii), the conclusion of the conjecture should be that $K$ is either an ellipsoid or a body of revolution. However, in case H(iii), ellipsoids  which are not bodies of revolution are excluded, (note that if an ellipsoid is not a body of revolution, it has $n$ axes of symmetry but no axis of revolution). Thus, the conclusion in case H(iii) should be that $K$ is a body of revolution.

In this paper, we focus on the cases H(ii) and H(iii) with some additional hypotheses, but we do not require the condition to hold for all sections of $K$,  just for sections passing through a fixed point. 
We need to introduce some definitions.

\begin{definition}
A point $p\in \inte K$ is said to be a \textit{Larman point} of $K$ if   for every hyperplane $\Pi$ passing through $p$ 
the section $\Pi\cap K$ has a $(n-2)$-plane of symmetry. 
\end{definition}


\begin{definition}
Let $p \in \inte K$ be a Larman point of $K$. We call $p$ a {\it revolution point} of $K$ if for every hyperplane $\Pi$ passing through $p$ 
the section $\Pi\cap K$ has a $(n-2)$-plane of symmetry which contains $p$.
\end{definition}

 As examples of Larman and revolution points, we note that if $c$ is the centre of an ellipsoid $E\subset \Rn$, which is not a body of revolution,  then $c$ is a revolution point of $E$.
 Furthermore, every point $p \neq c$ in the interior of the ellipsoid is a Larman point, but not a revolution point.  On the other hand, every point on the axis of a body of revolution  is a revolution point, while every point $p$ not on the axis is a Larman point.  We prove these facts in Corollary  \ref{ejemplos} (page 13). 

 With this terminology, we state the  following conjecture.
\begin{conjecture}\label{amaya}
Let $K\subset \Rn$, $n\geq 3$, be a convex body. Suppose that $p \in \inte K$ is a Larman  point of $K$ which is not  a revolution point of $K$. Then either $K$ is an ellipsoid or $K$ is a body of revolution.  
\end{conjecture}  

Observe that if a Larman point $p$ is also the centre of symmetry of $K$, then $p$ is a revolution point of $K$. Hence, Montejano's example of the convex hull of two discs is now excluded by the assumption that the Larman point $p$ is not a revolution point.


Our main results are the following theorems:

\begin{theorem}\label{tania}
Let $K\subset \Rn$ be a centrally symmetric strictly convex body with centre at $o$. Suppose that $K$ has a revolution point $p$, $p\not= o$. Then $K$ is a body of revolution whose axis is the line $L(o,p)$ passing through the points $o$ and $p$.
\end{theorem}

  In the next two theorems, we need an additional hypothesis regarding the existence of an auxiliary line. This  condition 
is, in a way, natural: If $K \subset \Rt$ is a body of revolution whose axis is the line $L$, and the point $p$ is not in $L$, then, on the one hand, $p$ is a Larman point of $K$ and, on the other hand, for every plane $\Pi$ passing through $p$, the section $\Pi \cap K$  has a line of symmetry passing through the point $\Pi \cap L$ (if $\Pi$ is not parallel to $L$). Note that in case (i) of Theorem \ref{sisi}, the fact that $p$ is the centre of symmetry of $K$ implies that $p$ is a revolution point, and the existence of the line $L$ excludes Montejano's counterexample.  

\begin{theorem}\label{sisi}
Let $K\subset \Rt$ be a centrally symmetric strictly convex body with centre at $o$. Let $L$ be a line such that $o \notin L$, let $\Omega$ be the plane containing $o$ and $L$, and let $p\in \Omega\setminus L$ be a  Larman point of $K$. 
Assume that for all planes $\Pi$ passing through $p$, the section $\Pi \cap K$ has a line of symmetry which intersects $L$  in (the case where the plane  is parallel to $L$, then the line of symmetry of the section of $K$ is assumed to be parallel to $L$). 
Then 
\begin{enumerate}[(i)]
  \item if $p=o$, then $K$ is a body of revolution,  
  
  \item if $p \neq o$ and the line $op$ is not perpendicular to $L$, then $K$ is a sphere.  
\end{enumerate}
\end{theorem}

We say that a line $L$ is an axis of symmetry of $K$ if, on the one hand, all sections of  $K$ by hyperplanes orthogonal to $L$ are centrally symmetric with center at a point in $L$, and on the other hand, all sections of $K$ by hyperplanes containing  $L$ have $L$ as a line of symmetry (\textit{i.e.}, given a hyperplane $H$, for every point $x\in K \cap H$, its reflection with respect to $L$ is also in $K\cap H$). With this terminology, we can state our next result.

\begin{theorem}
    \label{seven}
Let $K\subset \mathbb{R}^3$ be an origin symmetric, strictly convex body. Let $L$ be an axis of symmetry of $K$  containing the origin $o$, and let $p\in (\inte K )\setminus L$ be a Larman point of $K$. Suppose that for, every plane $\Pi$ passing through $p$, the section $\Pi \cap K$ has a line of symmetry which contains the point $\Pi\cap L$. Then $K$ is a body of revolution with axis $L$.
\end{theorem}


 As a corollary of Theorem \ref{tania}  we obtain:

\begin{corollary}\label{billy}
Let $K\subset \Rt$ be a centrally symmetric strictly convex body with centre at $o$. Suppose that $K$ has two distinct revolution points $p,q$ such that $p\not= o\not= q$ and $o$ does not belong to $L(p,q)$. Then $K$ is a sphere.  
\end{corollary}

 This corollary is an improvement on a result by Jer\'onimo-Castro, Montejano, and Morales-Amaya. 
 Inspired by the False Centre Theorem,  they give the following characterization of the sphere as the only body which contains a ``false axis of revolution''. 

{\bf Theorem}\cite{jemomo}:
   {\it   If a strictly convex body $K \subset \mathbb{R}^3$ contains a line $L$ such that all points in $L$ are revolution points, but $L$ is not an axis of revolution of $K$, then $K$ must be a sphere.  }

 {\bf Observations:} 
\begin{enumerate}
\item  In Corollary \ref{billy}, only two revolution points are needed, rather than the whole line needed in \cite{jemomo}. Note that, although the False Axis of Revolution Theorem does not assume that $K$ is centrally symmetric, the central symmetry actually follows from the hypotheses (see \cite[Lemma 2.4]{jemomo}). Therefore, the central symmetry in our Corollary \ref{billy} is not an additional assumption.

\item The proof of the False Axis of Revolution uses strict convexity by considering the set of extreme points  $t(x)$ of chords of $K$ whose center is $x\in \inte K$. For a strictly convex body $K$, the set $t(x)$ is contained in a plane. For a non strictly convex body this is no longer true.  Our Theorems \ref{tania}, \ref{sisi} and \ref{seven} also make use of this fact (see Lemma \ref{marina} in the present paper), which is why we need the strict convexity hypothesis. 

\end{enumerate}
 
Our next result is the dual version of Theorem \ref{sisi}, but now the conclusion is that $K$ is a body of revolution, and we no longer obtain the case where $K$ is a sphere.

\begin{theorem}\label{nerd}
Let $K\subset \Rt$ be a strictly convex body and let $L$ be a line. Suppose that every orthogonal projection of $K$ has a line of symmetry which intersects $L$ (if the plane of projection is parallel to $L$, then we assume that the line of symmetry of the projection of $K$ is parallel to $L$). Then $K$ is a body of revolution.  
\end{theorem}

 If $K$ is a body of revolution in $\Rn$, for $n \geq 4$, then all hypersections and projections of $K$ are $(n-1)$-dimensional bodies of revolution (see Remark \ref{boca} in Section \ref{revov}). Generalization H(iii) is the converse of this result. In Section \ref{revov}, we prove this converse when we know that all sections (or projections) through a fixed point $p$ are bodies of revolution,  under the additional assumption that  $K$ has a unique diameter.

\begin{theorem}\label{nariz}
Let $K\subset \mathbb{R}^n, n\geq 4$, be a convex body. Suppose that $K$ has a unique diameter $D$ and there is a point $p\in \mathbb{R}^n$ such that $p\notin D$ and for every hyperplane $\Pi$, passing through $p$, the section $\Pi \cap K$ is a $(n-1)$-body of revolution. Then $K$ is a body of revolution with axis $D$.
\end{theorem}

\begin{theorem}\label{nonis}
Let $K\subset \Rn$, $n\geq 4$, be a strictly convex body. Suppose that $K$ has a unique diameter $D$ and every orthogonal projection of $K$  is a $(n-1)$-body of revolution. Then $K$ is a body of revolution with axis $D$.
\end{theorem}

A more general result, without the diameter assumption, has recently been obtained by B. Zawalski \cite{zaw} in the case where  $K$ is an  origin symmetric convex body with boundary of class $C^3$,  and the point $p$ is the origin.

\newpage

\section{Definitions and auxiliary results}

 We refer to \cite[Chapter 0]{Gar} for the following  definitions involving convex bodies. A {\it body} in $\mathbb{R}^n$ is a compact set which is equal to the closure of its nonempty interior.  A {\it convex body} is a body $K$ such that for every pair of points in $K$, the segment joining them is contained in $K$. A convex body is {\it strictly convex} if its boundary does not contain a line segment. A body  $K$ is {\it origin symmetric} if whenever $x\in K$, it follows that $-x \in K$. A body $K$ is {\it centrally symmetric} if a translate of $K$ is origin symmetric, {\it i.e.} if there is a vector $c \in \mathbb{R}^n$ such that $K-c$ is origin symmetric. 

A {\it chord} of a convex body $K$ is any line segment in $K$ whose endpoints are on the boundary of $K$.  The {\it extreme points} of a chord are the endpoints of the line segment. A {\it diameter} of $K$ is a chord of maximal length. For each unit vector $\xi \in \mathbb{R}^n$, a chord parallel to $\xi$ of maximal length is called a {\it diametral chord} of $K$.  

For $n\geq 3$, we denote by $O(n)$ the \textit{orthogonal group}, \textit{i.e.}, the set of all the isometries of $\Rn$ that fix the origin. Let $K\subset \Rn$ be a convex body, let $\Pi$ be an affine hyperplane, and $p$ be a point in $\Pi$. We denote by $O(\Pi,p,n-1)$ the set of all isometries of $\Pi$ that fix $p$. When it is clear which affine  hyperplane $\Pi$ and point $p$ we are considering, we will abuse the notation and write $O(n-1)$ instead of $O(\Pi,p,n-1)$.

The section $\Pi \cap K$ is said to be \textit{symmetric} if there exists a non-trivial $\Omega\in O(n-1)$ such that
\[
\Omega(\Pi \cap K)=\Pi\cap K.
\]

\begin{definition}\label{rotation} Let $K \subset \Rn$ be a convex body, $n \geq 3$, and let $L$ be a line passing through the origin. We denote by $R_{L}:\Rn \rightarrow \Rn$ the element of $O(n)$ that acts as the identity on the line $L$, and sends $x$ to $-x$  on the hyperplane $L^{\perp}$. The line $L$ is said to be an axis of symmetry of $K$ if the following relation holds,
\[
R_{L}(K)=K.
\]

Note: When the line $L$ does not pass through the origin, we will abuse the notation and denote also by $R_L$ the function that acts as the identity on $L$, and sends $p+x$ to $p-x$ for every $p\in L$ and $x \in p+L^\perp$.

\end{definition}

We observe that if $L$ is an axis of symmetry of $K$, on the one hand, all sections of  $K$ by hyperplanes orthogonal to $L$ are centrally symmetric with center at $L$; on the other hand, all sections of $K$ by hyperplanes containing  $L$ have $L$ as a line of symmetry, \textit{i.e.}, $R_L$ restricted to each hyperplane containing $L$ is a reflection with respect to $L$.  Due to this property, the notion of axis of symmetry of a convex body will play an important role in the  proof of our Theorems. 

We will make frequent use of the following Remark in our proofs.

\begin{remark}\label{lineo}
 Let $K\subset  \mathbb{R}^n$ be an origin symmetric body. Let $H$ be a hyperplane passing through $o$. If the section $K\cap H$ has an axis of symmetry $L$, then $o \in L$.  (Similarly, if $K \cap H$ has a $k$ dimensional plane of symmetry $M$, then $o \in M$.)
\end{remark}

If $K\subset \Rt$ satisfies the property that for all the lines $L$ passing through the origin, $R_L(K)=K$, then $K$ is an sphere. Indeed, it follows from the hypothesis that all the sections of $K$ are centrally symmetric, and therefore $K$ satisfies condition (2) of \cite {B2}, implying that $K$ is an ellipsoid. But the unique ellipsoid with an infinite number of axes of symmetry, not all contained in a plane, is the sphere.
A stronger version of this result was proven in \cite{jemomo}, however it is not stated there as a theorem. Rather, it follows from the proof of 
Theorem 1 in \cite{jemomo}. For the convenience of the reader, we state it and prove it as a theorem here. 

 \begin{theorem} \label{aux} Let $K\subset \Rt$ be a convex body, $H$ be a plane, and $p\in H$ be a point. If every line $L$  contained in $H$ and passing through $p$ is an axis of symmetry of $K$, then $H$ is a plane of symmetry of $K$. Furthermore, $K$ is a  centrally symmetric body of revolution, whose axis is the line orthogonal to $H$ passing through $p$.
\end{theorem}

 \begin{proof}
{\it Step 1:} We first prove that $H$ is a plane of symmetry. Consider the section $K \cap H$. Since every line passing through $p$ is a line of symmetry for $K$, in particular, every line through $p$ is a line of symmetry for $K \cap H$. This means that $K \cap H$ is a disc with center $p$. 

Now consider a boundary point $x$ of $K$ not on $H$, and its projection $Px$ on $H$. Let $\ell_x$ be the line joining $x$ and $Px$, and let $d(x)$ be the distance between $x$ and $Px$.
The line joining $Px$ with $p$ is an axis of symmetry of $K$ and is perpendicular to $\ell_x$. Therefore, the other boundary point $x'$ on $\ell_x$ is at distance $d(x)$ from $Px$.
Since $x$ is arbitrary, we have shown that $H$ is a plane of symmetry of $K$.

{\it Step 2:} Consider the point $Px$ as in Step 1, and the point $Py$ on the line joining $Px$ and $p$, such that $d(Px,p)=d(Py,p)$.  Take the line $\ell_y$ perpendicular to $H$ passing through $Py$, and let $y,y'$ be the boundary points of $K$ on $\ell_y$. By step 1, we know that $y'$ is the reflection of $y$ with respect to $Py$.

We want to show that the distance  $d(y)$ from $y$ to $Py$ is equal to the distance $d(x)$ from $x$ to $Px$. Consider the diameter $L$ of the disc $K \cap H$ which is perpendicular to  the line $Px Py$. This diameter is also an axis of symmetry of $K$, and the plane containing $x,y,x'$ and $y'$ is perpendicular to $L$. Therefore, on this plane, $x$ gets reflected to a point $z' \in \ell_y$, and $y$ gets reflected to $z \in \ell_y$. If $z'\neq y'$ or $z \neq y$, we break the strict convexity of $K$. Therefore, $d(x)=d(y)$. 

Consider the circle centered at $p$ and passing through $Px$ and $Py$. We want to show that for every point $Pw$ on this circle, the boundary points $w$ and $w'$ on the line $\ell_w$ perpendicular to $H$ and passing through $Pw$ are at distance $d(x)$ from $Pw$. 

Let $\Lambda$ be the diameter bisecting $D$ and the line $Px Py$. Then the reflection of $x$ with respect to $\Lambda$ is on the line $\ell_w$ , and by strict convexity it must equal $w'$. Then, by considering the triangles with vertices $x, Px, p$ and $w', Pw, p$, we obtain that the distance from $w'$ to $Pw$ must be $d(x)$. This shows that the section of $K$ through $x$ parallel to $H$ is a disc, and therefore $K$ is a body of revolution with axis perpendicular to $H$ and passing through $p$.
\end{proof}

\begin{definition}\label{star}
A family of lines $\{L_1,..., L_n\}$ is called  a $n$-starline with vertex  $x_0$ if the lines $L_i$ lie on the same plane, they are concurrent at $x_0$, and the angle between two consecutive lines is $\frac{2\pi}{n}$.  
\end{definition}

If $L_1$ and $L_2$ are two lines with nonempty intersection, we denote by $\Omega (L_1,L_2)$  the set of all lines contained in the plane $\aff \{L_1,L_2\}$ and passing through the point  
$L_1\cap L_2$.

\begin{definition}
Let $L_1$ and $L_2$ be two axes of symmetry of the convex body $K$. The starline determined by $L_1$ and $L_2$, which will be denoted by $\Sigma (L_1,L_2)$, is the family of lines  
$\{T_n\}$ constructed in the following way: $T_1=L_1, T_2=L_2$, and, in general,
\[
 T_k=R_{T_{k-1}}(T_{k-2}).
\]
\end{definition}
 We observe, on the one hand, that each line in the family $\{T_n\}$ is an axis of symmetry of $K$ and, on the other hand, that $T_i\subset \Omega(T_1,T_2)$ for all $i$. 
 
\begin{proposition}\label{escandalo}
Let $L_1$ and $L_2$ be two axes of symmetry of the convex body $K\subset \Rt$. If the angle between $L_1$ and $L_2$ is $\frac{2\pi}{n}$, for some integer $n$, then $\Sigma (L_1,L_2)$ is a $n$-starline for some integer $n$; otherwise, $\Sigma (L_1,L_2)$ is a dense set in $\Omega (L_1,L_2)$ and, consequently, $\Sigma (L_1,L_2)\cap K$ is a circle.
\end{proposition}
The proofs of the following results and Lemmas are straightforward, so we will only give the proof of Lemma \ref{contento}. 
\begin{itemize}
\item[I.] Let $\Phi$ be a planar convex body and let   $\{L_1,..., L_n \}$ be the collection of all its lines 
of symmetry. Then,  $\{L_1,...,L_n \}$ is an $n$-starline. 
 
\item[II.] Let $K \subset \Rt$ be a convex body and let $\{H_i \}$ be a sequence of hyperplanes that intersect $\inte K$. Suppose that $H_i \rightarrow H$, $L_i \subset H_i$ is an axis of symmetry (respectively, $p_i \in H_i$ is a center of symmetry) of $H_i \cap K$, and $L_i \rightarrow  L$ ($p_i \rightarrow  p$); then, $L$ is an axis of symmetry (respectively, $p$ is a center of symmetry) of $H \cap K$. 
 
\item[III.] Let $\{K_i\}$ be a sequence of planar convex bodies such that  $K_i \rightarrow K$ and, for every $i\in\mathbb{N}$, the body $K_i$ has two axes of symmetry determining an angle $\theta_i$. If 
$\lim_{i\to\infty}\theta_i=0$ then  $K$ is a disc.
\end{itemize}

\begin{lemma}\label{contento}
Let $K\subset \Rt$ be a convex body. Suppose that $\{L_n\} \subset \Rt$ is a sequence of axes (hyperplanes) of symmetry of $K$, and $L$ is a line (hyperplane) such that $L_n \rightarrow L$. Then $L$ is an axis (hyperplane) of symmetry of $K$.
\end{lemma}

\begin{proof}
Since $L_n \rightarrow L$, by  \cite[Theorem 1.8.7]{sch}, for all $q\in L \cap K$, there exists a sequence $q_n \in L_n \cap K$ such that $q_n \rightarrow q$. We denote by $\Gamma_n$  the orthogonal hyperplane to $L_n$ passing through $q_n$, and by $\Gamma$ the orthogonal hyperplane  to $L$ passing through $q$. Since $L_n$ is an axis of symmetry of  $K$, $\Gamma_n \cap K$ is centrally symmetric with center at $q_n$. By virtue of the fact that $L_n \rightarrow L$ and $q_n \rightarrow q$, we have $\Gamma_n \rightarrow \Gamma$. Thus $\Gamma_n \cap K \rightarrow \Gamma \cap K$. From II, it follows that $\Gamma \cap K$ is centrally symmetric with center at $q$. Thus $L$ is an axis of symmetry of $K$.
\end{proof} 

\begin{lemma}\label{chinita}
Let $M\subset \Rt$ be a planar convex body contained in the plane $\Pi$ and let $a,b$ be two points in $\Pi$. 
Let $\{\Pi_k\}_{k=1}^{\infty}$ be a sequence of planes, $\Pi_k$ containing $L(a,b)$ and let $\{M_k\}_{k=1}^{\infty}$ be a sequence of planar convex bodies, 
$M_k \subset \Pi_k,k=1,2..$. Suppose that, for each $k$, $M_k$ has two lines of symmetry $l_k,m_k$ so that $a\in l_k, b\in m_k$ and 
\begin{eqnarray}\label{sombras}
M_k \rightarrow  M.
\end{eqnarray}
Then $M$ has two lines of symmetry $l,m$ such that $a\in l, b\in m$ and
\begin{eqnarray}\label{grey}
l_k \rightarrow  l \textrm{ and }m_k \rightarrow m.
\end{eqnarray}
\end{lemma} 
As a corollary of Lemma \ref{chinita} we have
\begin{eqnarray}\label{grecia}
l_k \cap m_k \rightarrow  l \cap m.
\end{eqnarray}

\section{Sections and projections of bodies of revolution}\label{revov}

In this section we obtain converses of the following remark,   under the additional assumption that the body $K$ has a unique diameter. 

\begin{remark}\label{boca}If $K$ is a body of revolution in $\Rn$,  $n \geq 4$, then all hypersections and projections of $K$ are $(n-1)$-dimensional bodies of revolution. 
\end{remark}

For completeness, we begin by giving a proof of  Remark \ref{boca}.

\begin{proof}
 \textit{(i)  Hypersections:} We denote by $L$ the line containing the axis of revolution of $K$. Let $\Pi$ be a hyperplane such that $\Pi \cap \inte K\not= \emptyset$, and let $u$ be a unit normal vector of $\Pi$. Let $\Delta$ be a two dimensional plane containing $u$ and $L$, \textit{i.e.}, $\Delta$ is perpendicular to $\Pi$ and $L\subset \Delta$. We will show that $\Pi \cap K$ is a body of revolution with axis $\Pi \cap \Delta$. 

Let $\Gamma$ be a hyperplane perpendicular to $L$ such that $(\Gamma \cap \Pi) \cap \inte K \not= \emptyset$. Since $K$ is a body of revolution $\Gamma \cap K$ is a $(n-1)$-Euclidean ball  with centre at $q:=\Gamma \cap L$. The line $\Gamma \cap \Delta$ passes through $q$ because $L \subset \Delta$, so it is  a diameter of the ball  $\Gamma \cap K$  (See Figure \ref{rev}). Consequently $(\Gamma \cap \Pi)\cap K$ is a $(n-2)$-Euclidean ball  with centre on the line $\Gamma \cap \Delta$ and also on the line $\Pi \cap \Delta$. Since $\Gamma$ is arbitrary, it follows that $\Pi \cap K$ is a body of revolution whose axis is the line $\Pi \cap \Delta$.

 \begin{figure}
\centering
\includegraphics [width=5.5in] {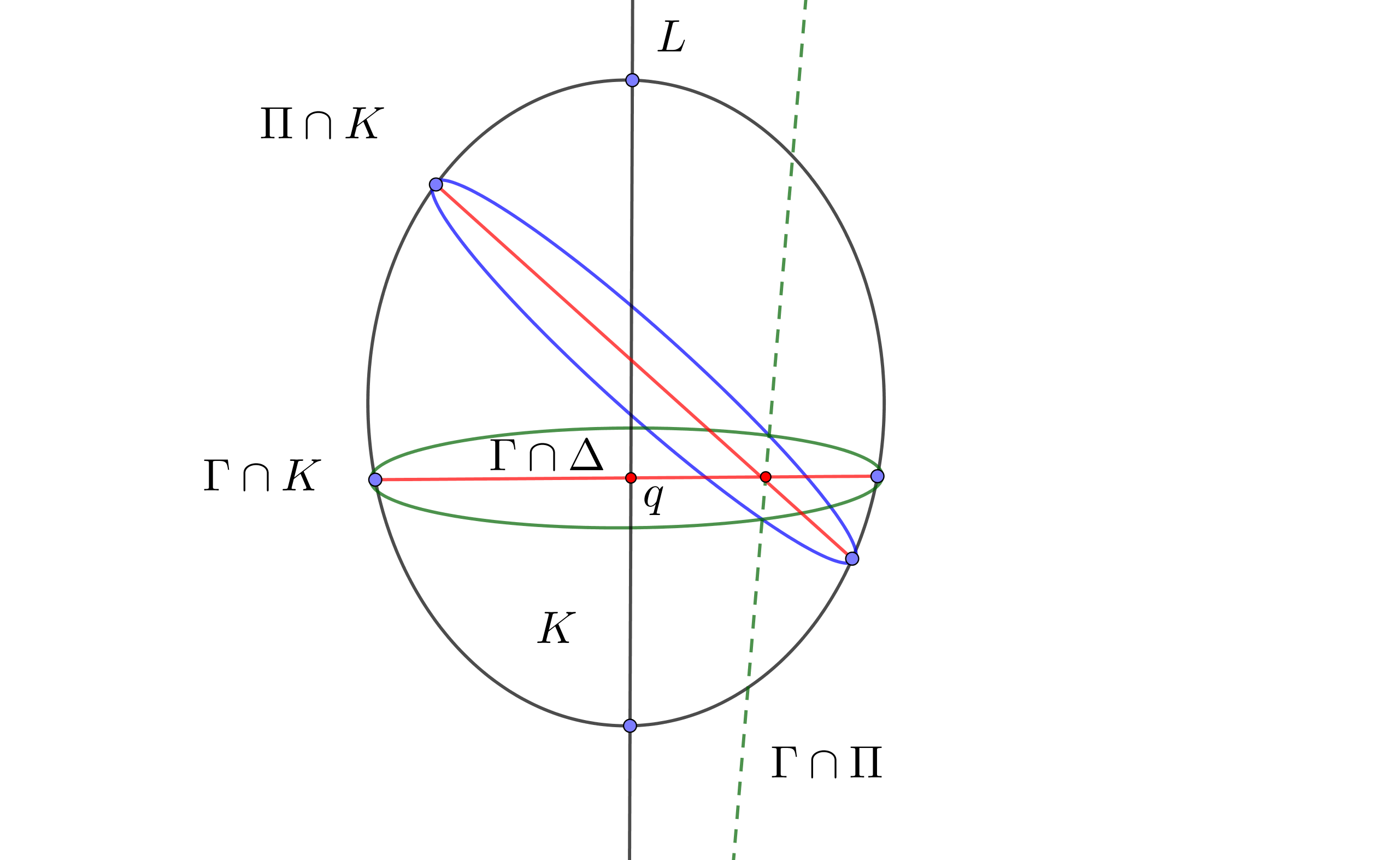} 
\caption{A section of a body of revolution is a body of revolution.}
\label{rev}
\end{figure}

 \textit{(ii) Projections: } Let $\Pi$ be a hyperplane and let $u$ be a unit normal vector of $\Pi$. We denote by $\phi_u:\Rn\rightarrow \Pi$ the orthogonal projection parallel to $u$, and by $M$ and by $K_u$ the sets $\phi_u(L)$, $\phi_u(K)$, respectively. We will prove that $K_u$ is a body of revolution with axis $M$, by showing the every $(n-2)$-section of $K_u$ perpendicular to $M$ is a sphere with centre on $M$.

Let $N\subset \Pi$ be an affine subspace of dimension $n-2$ orthogonal to $M$, and such that $N\cap \inte K_u\not=\emptyset$. Let $\Gamma=\phi^{-1}_u(N)$ and $\Delta=\phi^{-1}_u(M)$.  By (i), $\Gamma \cap K$ is a body of revolution with axis $\Gamma\cap \Delta$ (see Figure \ref{otra}). Notice  that $\Gamma\cap \Delta$ is parallel to $u$. Thus $\phi_u(\Gamma \cap K)=N\cap K_u$ is a sphere with centre at $M$. Since $N$ is arbitrary, it follows that $K_u$ is a body of revolution with axis of revolution $M$.
\end{proof}

\begin{figure}
\centering
\includegraphics [width=5.5in] {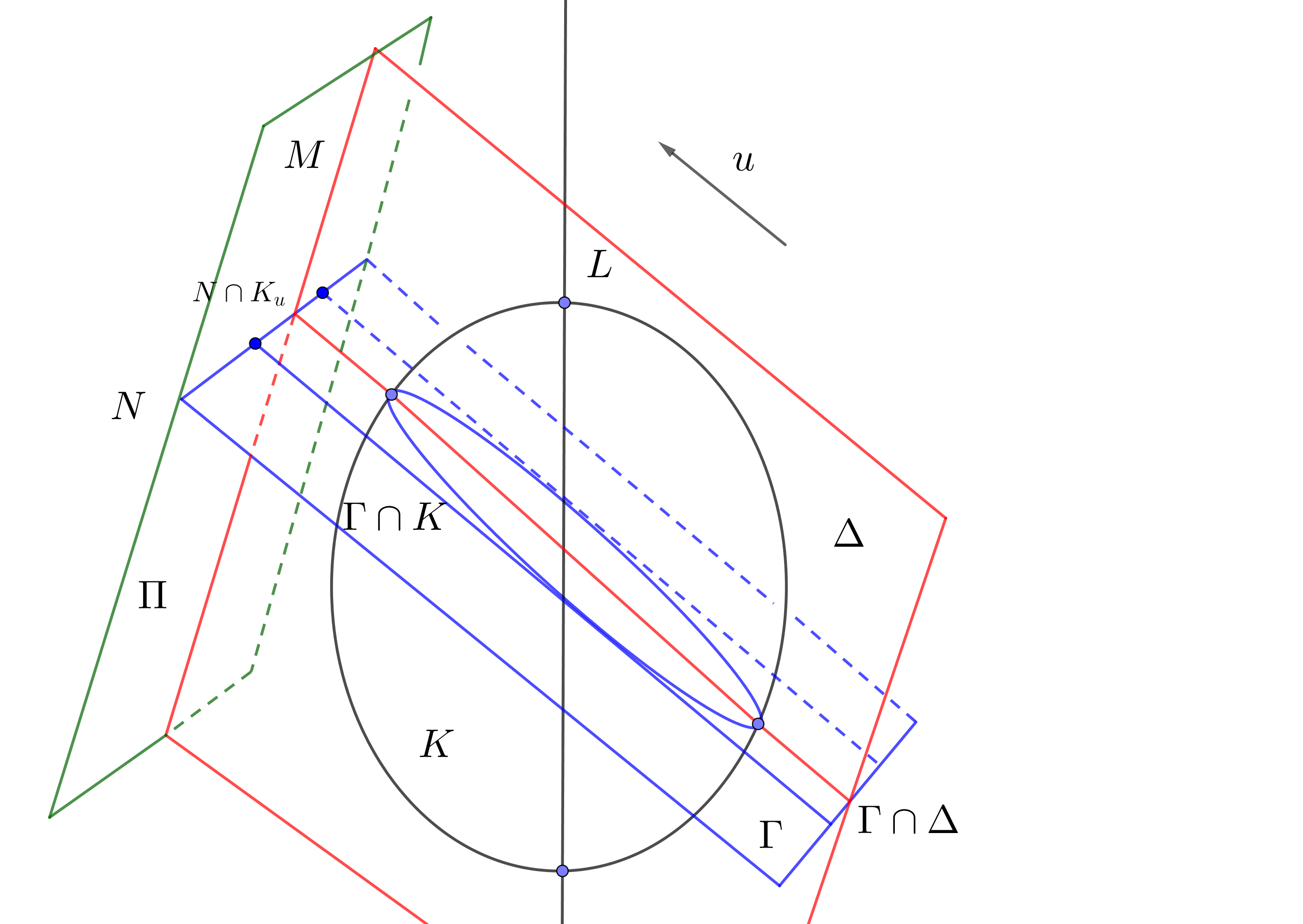} 
\caption{A projection of a body of revolution is a body of revolution.}
\label{otra}
\end{figure}

Now we prove the converse for sections, under the assumption that $K$ has a unique diameter. 
  

{\bf Proof of Theorem \ref{nariz}.}

\begin{proof}
Fix an arbitrary point $x \in D$ as the origin of coordinates.   Let $w_0$ be a unit  vector parallel to $D$. We will show that $w^{\perp}_{0} \cap K$ is a sphere with centre at $x$. 

Since $p\notin D$, there exists a unit vector $w_1$ in $\aff\{w_0,p\}$  such that $w_1\perp w_0$.   We choose   any orthonormal basis  of  $w^{\perp}_{0}$ which contains $w_1$ as one of its vectors, {\it i.e. }, $w^{\perp}_{0}=span(w_1, w_2,w_3,...,w_{n-2},w_{n-1})$.

For each  choice $\{w_{i_1},\ldots,w_{i_{n-3}}\}$ of $n-3$ vectors from the set 
\[
\{w_2,w_3,...,w_{n-2}, w_{n-1} \} \subset w^{\perp}_{0},
\]
we consider the hyperplane
\[
\widetilde{\Pi}:=\Pi(w_{i_1},\ldots,w_{i_{n-3}})=p+\textrm{span}\{ w_0,w_1,w_{i_1},\ldots,w_{i_{n-3}}\},
\] passing through $p$. By hypothesis, we have that  $\widetilde{\Pi} \cap K$ is a $(n-1)$-body of revolution. 

Since $D$ is the unique diameter of  $\widetilde{\Pi}  \cap K$, $D$ must be its axis of revolution.   
Thus, 
the section 
\[
\aff \{w_1, w_{i_1},\ldots,w_{i_{n-3}}\} \cap (\widetilde{\Pi} \cap K)
\] 
is an $(n-2)$-dimensional Euclidean ball with centre at $x$. Repeating the argument for every choice of $\{w_{i_1},\ldots,w_{i_{n-3}}\}$ of $n-3$ vectors from the set \\
$\{w_2,w_3,...,w_{n-2}, w_{n-1} \}$,
we conclude that $
w^{\perp}_{0} \cap K$ 
is a  sphere with centre at $x$. Finally, by the arbitrariness of the choice of the point $x\in \inte D$, the theorem follows.
\end{proof}

\subsection{Proofs of Theorems \ref{nerd} and \ref{nonis}} 

In order to prove the converse for projections, we first need to establish some lemmas. 
For $n\geq 2$, let $W\subset \Rn$ be a set, and let $\Pi \subset \Rn$ be a hyperplane with unit normal vector $u$ (if $\Pi$ passes through the origin $o$, we denote it by $u^{\perp}$). Let $\phi_u:\Rn\rightarrow \Pi$ be the orthogonal projection parallel to $u$, and by $W_u$ the set $\phi_u(W) \subset \Pi$.  
 If the hyperplane of projection is not specified, then we will assume that the projection is onto the hyperplane $u^{\perp}$.

\begin{lemma}\label{gretel}
Let $K\subset \Rt$ be a strictly convex body and let $L$ be a line. Suppose that, for every unit vector $u$ orthogonal to $L$, $K_u$ has an axis  of symmetry parallel to $L$. Then $K$ has an axis of symmetry parallel to $L$.  
\end{lemma}

 In the proof of Lemma \ref{gretel} we use  the following lemma, whose proof is immediate.
\begin{lemma}\label{alma}
Let $W\subset \Rd$ be a planar convex body and let $p\in \Rd$. If, for every unit vector $u$, the point $\phi_u(p)$ is the midpoint of $W_u$, then $W$ is centrally symmetric with centre at $p$. 
\end{lemma}

\textbf{Proof of Lemma \ref{gretel}.}
We choose a coordinate system such that the origin is contained in $L$. 
Let $\Pi_1,\Pi_2$ be supporting planes of $K$ perpendicular to $L$, making contact with $\partial K$ at the points $E$ and $F$, respectively. Let $u$ be a unit vector perpendicular to $L$, $K_u=\phi_u(K)$ be the projection of $K$ onto $u^\perp$. We have that $L=\phi_u(L)$, since $L\subset u^{\perp}$ for all $u$ perpendicular to $L$. We denote $\phi_u(E)=A$ and $\phi_u(F)=B$. 

By hypothesis, there exists a line of symmetry $M$ of $K_u$  that is parallel to $L$. Since $L$ and  $u^{\perp}$ 
are parallel, the line $M$ is perpendicular to $\Pi_1\cap u^{\perp}$ and to $\Pi_2\cap u^{\perp}$.  By the strict convexity of $K_u$, it follows that $M=L(A,B)$. 
Thus, we conclude that the orthogonal projection of the diametral chord $EF$ of $K$ onto $u^{\perp}$ is perpendicular to $\Pi_1$ and $\Pi_2$. By the arbitrariness of $u$ we obtain that $EF$ is perpendicular to $\Pi_1$ and $\Pi_2$, and thus parallel to $L$. 
Therefore, for the remainder of the proof, we will assume that $L$ is the line containing the segment $EF$.

 We will prove that for every plane $\Gamma$ perpendicular to $L$ with $\Gamma \cap \inte K\not= \emptyset$, the section $\Gamma \cap K$ is centrally symmetric with centre at  $\Gamma \cap L$. In order to prove this, we will show that the section $\Gamma \cap K$ and the point $\Gamma \cap L$ satisfy the conditions of Lemma \ref{alma}.

Let $w$ be a unit vector perpendicular to $L$. By hypothesis, $K_w$ is symmetric with respect to a line $M$ parallel to $L$. By the argument of the previous  paragraph, $\phi_w(L)=M$. 
Thus the chord $\phi_w(\Gamma)\cap K_w$ has its midpoint at 
$\phi_w(\Gamma)\cap M=\phi_w(\Gamma)\cap \phi_w(L)$. It follows that 
$\phi_w(\Gamma)\cap \phi_w(L)=\phi_w(\Gamma)\cap L=\phi_w(\Gamma \cap L)$, \textit{i.e.}, $\phi_w(\Gamma)\cap K_w$ has its midpoint at $\phi_w(\Gamma \cap L)$. By Lemma \ref{alma}, 
$\Gamma \cap K$ is centrally symmetric with centre at $\Gamma \cap L$. Varying $\Gamma$, always perpendicular to $L$ and such that $\Gamma \cap \inte K\not= \emptyset$, we conclude that $EF$ is an axis of symmetry of $K$.  
\qed

With almost the exact same arguments, one can also prove the following Lemma.

\bigskip
\begin{lemma}\label{damaris}
Let $K\subset \Rn$ be a strictly convex body, $n> 3$ and let $L$ be a line. Suppose that,  for every unit vector $u$ orthogonal to $L$, $K_u$ has an axis  of revolution  parallel to $L$. Then  $K$ has an axis of revolution parallel to $L$. 
\end{lemma}

    

Now are are ready to prove Theorems \ref{nerd} and \ref{nonis}.
  

{\bf Proof of Theorem \ref{nerd}.}
 
\begin{proof}  By hypothesis, the projection of $K$ on any plane parallel to $L$ has a line of symmetry parallel to $L$. Hence, by Lemma \ref{gretel}, $L$ is an axis of symmetry of $K$.

On the other hand, if $\Phi\subset \mathbb{R}^2$ is a convex body with a line of symmetry $W$, and $q$ is a point contained in $W$, then 
in order to determine $W$ we must find two parallel supporting lines $L_1$, $L_2$ of $\Phi$, such that the distance between $q$ and $L_1$ is equal to the distance between $q$ and $L_2$ (since $W$ is equidistant from $L_1$ and $L_2$). It is clear that if $q$ does not belong to $\Phi$ such couple $L_1,L_2$ is unique. Hence, if for a unit vector $u$, $M$ is a line of symmetry of $K_u$ passing through the point $u^{\perp}\cap L$, then there exists   a pair of supporting planes $\Delta_1$, $\Delta_2$ of $K$, parallel to $u$ and $M$, and  such that $M$ is equidistant from $\Delta_1$ and $\Delta_2$. Let $w$ be a unit vector perpendicular to $u$ and $L$ and let $\Pi_1, \Pi_2$ two parallel supporting planes of $K$ perpendicular to $w$. Since $L$ is an axis of symmetry of $K$,  $R_L(\Pi_1)=\Pi_2$ and $L$ is equidistant from $\Pi_1$ and $\Pi_2$. Thus, the supporting parallel lines $\phi_u(\Pi_1)$, $\phi_u(\Pi_2)$ of $K_u$ are at the same distance from the point $u^{\perp}\cap L\in M$. Hence, by the aforesaid, $M\subset w^{\perp}$ (the origin of a system of coordinates is in $L$). Now, varying $u$ while keeping $w$ fixed, we conclude that $w^{\perp}$ is a plane of symmetry of $K$ (all the orthogonal projections of $K$ in direction perpendicular to $w$ have a line of symmetry in $w^{\perp}$). Given that all the planes containing $L$ are planes of symmetry of $K$, $K$ is a body of revolution with axis $L$.
\end{proof}

{\bf Proof of Theorem \ref{nonis}.}

\begin{proof} We denote by $L$ the line generated by $D$. We will show that $K$ and $L$ satisfy the conditions of Lemma \ref{damaris} and, consequently, we will conclude that $K$ is a body of revolution with axis $L$. Let $u$ be a unit vector orthogonal to $L$ and we take a system of coordinates such the origin is in $L$. Thus $\phi_u(L)=L$ and, since $D$ is a binormal (\textit{ i.e.}, the normal vectors of $K$ at the endpoints of $D$ are parallel to $D$), and $u$ is orthogonal to $L$, $\phi_u(D)=D$. Therefore, $\phi_u(D)$ is the unique diameter of $K_u$. We claim that $L$ is the axis of revolution  of $K_u$. Otherwise, $K_u$ would not have a unique diameter. Hence $K$ and $L$ satisfy the conditions of Lemma \ref{damaris}. 
\end{proof}

\section{Proof  of Theorem \ref{tania}} 

  We begin by proving Lemma \ref{marina}, a crucial fact needed in the proofs of Theorems \ref{tania}, \ref{sisi} and \ref{seven}. This is the main reason why our Theorems need the strict convexity hypothesis, as discussed in Observation 2 on page 4.   

 For $x\in \inte K$ we denote by $C(x)$ the family of chords of $K$ whose midpoint is $x$. Let $t(x)$ be the locus of the extreme points of the chords in $C(x)$. Lemma \ref{marina} shows that, if $K$ is origin symmetric, strictly convex and has an axis of symmetry $\Lambda$, then  the set $t(x)$ lies on the hyperplane $H$ perpendicular to $\Lambda$ passing through $p$. The heuristic idea is as follows: Suppose that a chord whose midpoint is $p$ is not contained in $H$. By reflecting the endpoints of this chord around $o$ (using the central symmetry) and around $\Lambda$ (using that it is an axis of symmetry), one obtains three boundary points of $K$ that line on the same line. This contradicts the strict convexity.

\begin{lemma} \label{marina}
Let $K\subset \mathbb{R}^n$ be a centrally symmetric strictly convex body with centre at $o$ and let $p$ be a point, $p\not=o$. Suppose that $K$ has an axis of symmetry $\Lambda$ and $p\in \Lambda$. Let $H$ be the hyperplane perpendicular to $\Lambda$ that passes through $p$. Then $t(p)=H \cap \partial K$.
\end{lemma}

\begin{proof}

Suppose that there exists a chord $AB\in C(p)$ such that $AB$ is not contained in the hyperplane $H$.  We consider the two dimensional plane containing the line $\Lambda$ and the chord $AB$, and on that plane take a system of coordinates $(x_1,x_2)$  such that $p$ is the origin and  $\Lambda$ is the $x_1$-axis. Since $K$ is centrally symmetric and $\Lambda$ is an axis of symmetry, we have that the centre  $o$ lies on $\Lambda$, see Figure \ref{choco}. 
We denote by $A',B'$ the points on the plane $x_1x_2$ which are reflections of $A,B$ with respect to the line $\Lambda$, and by $\bar{A}, \bar{B}$ the images of $A,B$ under the central reflection with respect to $o$.

Let $(r,0), (a_1,a_2), (b_1,b_2)$ be the coordinates of $o$, $A$ and $B$, respectively, where $r<0$. 
Since we are assuming that the chord $AB$ is not contained in the hyperplane $x_1=0$, exchanging if needed the points $A,B$, we have only two possible cases:  $a_1<r$ and $b_1>r$, or $a_1>r$ and $b_1>r$.

In the case where $a_1<r$ and $b_1>r$, 
we denote by $M$ the half-plane $\{(x_1,x_2) :x_1>r\}$. Thus $\bar{A}\in M$ and, in particular, $\bar{A}\not=A'$. Observe that the lines $L(A,B')$, $L(A',B)$ are parallel to $\Lambda$ and equidistant from it.  Hence, $\bar{A}\in L(A',B)$ (see Figure \ref{choco}).  On the other hand, since  $p\not=o$, it follows that $\bar{A}\not=B$. Consequently, the line $L(A',B)$ contains three different boundary points of $K$, which contradicts the strict convexity of $K$.

 Similarly, 
the case $a_1>r$ and $b_1>r$ is impossible. Let $S\partial (K, \Lambda)$ denote the shadow boundary of $K$ in the direction of  $\Lambda$, which is equal to 
\begin{eqnarray}\label{sukim}
S\partial (K, \Lambda)=\{x_1=r\} \cap \partial K,
\end{eqnarray}
since $K$ is origin symmetric and $\Lambda$ is an axis of symmetry. 
If  $a_1>r$ and $b_1>r$, there would exist a point $x\in S\partial(K,\Lambda)$ such that $x\in M$, which would contradict (\ref{sukim}). 

Thus, we conclude that 
\begin{eqnarray}\label{linkey}
t(p)=\{x_1=0\}\cap \partial K.
\end{eqnarray}
\end{proof}  
 
\begin{figure}[H]
\centering
\includegraphics [width=4in] {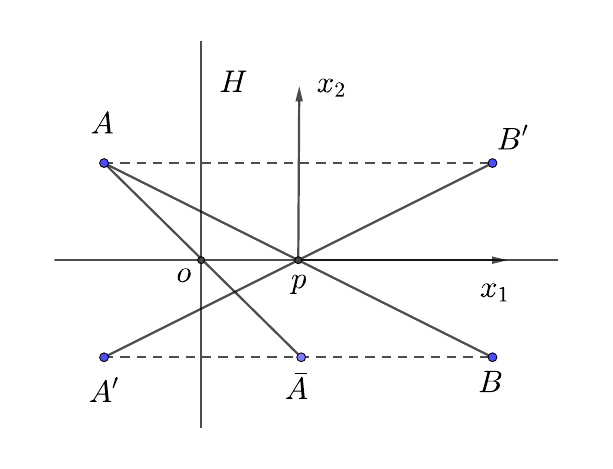} 
\caption{$t(p)$ is a planar curve.} 
\label{choco}
\end{figure}


  Before proving Theorem \ref{tania}, we will outline the main ideas of the proof in the three dimensional case:  The first step is to prove that $L(o,p)$ is an axis of symmetry of $K$. Consider the pencil of planes containing the line $L(o,p)$. Since $p$ is a revolution point, for each such plane $H$, the section $K\cap H$ has a line of symmetry that passes through $p$. But by Remark \ref{lineo}, this line of symmetry also passes through $o$, and hence it is $L(o,p)$. 

In the second step, we consider any line $L'$ passing through the point $p$ and perpendicular to $L(o,p)$. For any plane on the pencil containing $L'$, we show that the line of symmetry given by the hypothesis is perpendicular to $L'$. Hence, the union of all such lines for the pencil of planes containing $L'$ form a plane of symmetry of $K$. Combining Steps 1 and 2 gives that $L(o,p)$ is not only an axis of symmetry but also an axis of revolution.

{\bf Proof of Theorem \ref{tania}.}

\begin{proof}

We take a system of coordinates $(x_1,x_2,\dots,x_n)$ of $\Rn$ such that $p$ is the origin, the line $L(o,p)$ generated by $o$ and $p$ corresponds to the axis $x_n$, and $o$ has coordinate $(0,0,\ldots,r)$, $r>0$. Let $H$ be a hyperplane containing the line $L(o,p)$. By hypothesis, $K\cap H$ has an $(n-2)$ plane of symmetry that passes through $p$ (since $p$ is a revolution point) and through $o$ (by Remark \ref{lineo}). Therefore, the line $L(o,p)$ is an axis of symmetry of $K$, since it is contained in all such $(n-2)$ planes.

Now fix any line $L$ in $\Rn$ which is perpendicular to the $x_n$ axis. By changing the coordinate system, we may assume without loss of generality that $L$ lies on the $x_1 x_2$ plane and makes an angle $\theta \in [0, \pi]$ with the $x_1$ axis.  Let $\Gamma$ be a hyperplane containing $L$ and let $M$ be an $(n-2)$ plane of symmetry of $\Gamma \cap K$ passing through $p$. We claim that $M$ is perpendicular to $L$. On the contrary, assume that $M$ is not perpendicular to $L$ and let $\phi:\Gamma \rightarrow \Gamma$ be the reflection on $\Gamma$ with respect to the plane $M$. Since $M$ is not perpendicular to $L$, the relation $\phi(L\cap K)\not= L\cap K$ holds. Thus $\phi(L\cap K)$ is not contained in $\{(x_1,x_2,\ldots,x_n)\in \Rt : x_n=0\}$ which contradicts Lemma \ref{marina}). This shows that $M$ must be perpendicular to $L$. Varying $\Gamma$ among the pencil of hyperplanes containing $L$, it follows that the union of all such $(n-2)$ planes of symmetry $M$ is the hyperplane $\Pi(L)$ perpendicular to $L$ and passing through $p$. Consequently, $\Pi(L)$ is a hyperplane of symmetry of $K$ for any line $L$ perpendicular to $L(o,p)$ and passing through $p$. Thus, $K$ is a body of revolution with axis $L(o,p)$, completing the proof of Theorem \ref{tania}.

\end{proof}

As a corollary, we identify all the Larman and revolution points of  ellipsoids and strictly convex, origin symmetric bodies of revolution.

\begin{corollary}\label{ejemplos}

(i)  Let $E \subset \mathbb{R}^n$ be an ellipsoid which is not a body of revolution, and let $o$ be its center. Then $o$ is a revolution point of $E$, and any interior point $p \in E$ such that  $p\neq o$ is a Larman point, but is not a revolution point of $E$. 

(ii) Let $K\subset \mathbb{R}^n$ be an origin symmetric, strictly convex body of revolution which is not a ball. Then any  interior point $p \in K$ that is on the axis of revolution is a revolution point, while any  $p \in K$ not on the axis of revolution is a Larman point which is not a revolution point of $K$. 
\end{corollary}

\begin{proof}
    
Any hyperplane section of an ellipsoid in $\mathbb{R}^n$ is an $(n-1)$ dimensional ellipsoid, and any hyperplane section of a body of revolution is an $(n-1)$ dimensional body of revolution (see Remark \ref{boca}). Hence, any hyperplane section of an ellipsoid or a body of revolution has an $(n-2)$ plane of symmetry, which implies that any interior point is a Larman point.

(i) By Remark \ref{lineo}, the center of the ellipsoid $E$ is a revolution point. Let us assume that an interior point $p \neq o$ is a revolution point.  The ellipsoid $E$ has $n$ axes of symmetry, none of which are axes of revolution. First, consider the case in which $p$ is not on an axis of symmetry of $E$. Let $\Pi$ be a hyperplane containing the line $L(o, p)$. On the one hand, since $p$ is a revolution point of $E$, there exists an $(n-2)$ plane  of symmetry $W$ of $\Pi \cap E$ passing through $p$. On the other hand, by Remark \ref{lineo},  $W$ must also pass through $o$. Thus 
$L(o, p) \subset W$. Since this is true for any $\Pi$ containing $L(o,p)$, 
it follows that $L(o,p)$ is an axis of symmetry of $E$. This contradicts the choice of $p$. 

Now we assume that $p$ is on an axis of symmetry $I$ of $E$. Since $E$ is strictly convex, by Lemma \ref{marina}, the equality $t(x)=H\cap E$ holds, where $H$ is the plane perpendicular to $I$ and passing through $x$. By an analogous argument as the one used in the proof of Theorem \ref{tania}, it follows that $E$ is a body of revolution with axis $I$. This contradicts the assumption that $E$ is not a body of revolution.

(ii) Let $K\subset \mathbb{R}^n$ be a strictly convex, origin symmetric body of revolution with axis of revolution $L$, and assume that $K$ is not a ball. Assume that $p\in L$. Since $L$ must be contained in the $(n-2)$ plane of symmetry of any hyperplane section $K\cap \Pi$ where $\Pi$ contains $L$, then $p$ is on the $(n-2)$ plane of symmetry. If $\Pi$ does not contain $L$, the  $(n-2)$ plane of symmetry must contain the line of symmetry of the section $\Pi \cap K$, which is the line passing through $p$ and the point $\Pi \cap L$ (see Figure \ref{rev} and the proof of Remark \ref{boca}). In both cases, it follows that $p$ is a revolution point. 
On the other hand, if we assume that $p \notin L$ is a revolution point, it follows from Theorem \ref{tania} that $K$ is a body of revolution with axis the line $L(o,x)$, i.e., $K$ is a body of revolution with respect two different axis of revolution. Thus $K$ is a ball, contradicting our hypothesis.
\end{proof} 


\section{Proof of Theorems  \ref{sisi} and \ref{seven}}


 We will first outline the main ideas of the proof of Theorem \ref{sisi}, which requires several auxiliary Lemmas.

Our first Lemma  \ref{liz} shows that, under the hypotheses of Theorem \ref{sisi}, if $o=p$ then any line passing through $o$ and contained in the plane $\Omega$ is an axis of symmetry of $K$. Then, by Theorem \ref{aux}, $K$ is a body of revolution. This concludes the proof of case (i) of Theorem \ref{sisi}. 

In case (ii), when $o \neq p$, Lemma \ref{liz} yields that the line $\Lambda:=L(o,p)$ is axis of symmetry of $K$. By Remark \ref{lineo}, $\Lambda$ is the unique axis of symmetry of $K$ that contains the point $p$.  

Next, we consider the line $L(\theta)$ contained in the plane $\Omega$, passing through $p$ and making an angle $\theta$ with $\Lambda$, and the family of planes in $\mathbb{R}^3$ that contain the line $L(\theta)$. In each of these planes, the corresponding section of $K$ has either one or two lines of symmetry: the first line is given by the Larman point hypothesis, and the second line (which may coincide with the first one) is obtained by reflection of this line on the axis of symmetry $\Lambda$. Lemma \ref{kenny} shows that these two lines must be distinct for almost every plane in the family. 

Having now two different lines of symmetry on almost all plane sections, Lemmas \ref{simone}, \ref{caminante} and \ref{cometa} use topological arguments to conclude that the corresponding sections of $K$ by planes in the family containing $L(\theta)$ must be discs. Finally, Lemmas \ref{astrud} and \ref{sph} allow us to conclude that $K$ is a ball.

\begin{lemma}\label{liz}
Let $K \subset \Rt$ be a centrally symmetric strictly convex body with centre at $o$, let $L$ be a line and let $p$ be a Larman point of $K$. Suppose that $o\notin L$, $p\in \Omega \backslash L$ and, for all planes $\Pi$ through $p$, one line of symmetry of $\Pi \cap K$ passes through $\Pi \cap L$. (In the case where $\Pi$ is parallel to $L$, then the line of symmetry of $\Pi \cap K$ is assumed to be parallel to $L$). Then the line $\Lambda$ is an axis of symmetry of $K$. 

Furthermore, if $o=p$ (that is, if $K$ is centrally symmetric with respect to the Larman point), then $K$ is a body of revolution with an axis perpendicular to the plane $\Omega$ and passing through $o$.
\end{lemma}

\begin{proof}
First we consider the case $o\not= p$. Since $p\in \Omega$, then either $\Lambda \cap L \not= \emptyset$ or $\Lambda$ and $L$ are parallel. Let us assume first that $\Lambda \cap L \not= \emptyset$ and we denote by $q$ the intersection $\Lambda \cap L \not= \emptyset$. We denote by $\cal{F}$ the bundle of planes containing $\Lambda$, and we will  show that $\Lambda$ is an axis of symmetry of $K$, by proving that, for every plane $\Pi\in \cal{F}$, the section $\Pi \cap K$ has $\Lambda$ as line of symmetry. Let $\Pi \in \cal{F}$. Since $p$ is a Larman point of $K$, there is at least one line of symmetry of $\Pi \cap K$ and, by hypothesis, one of these lines (say $W$) passes through $\Pi \cap L$. But the only point of intersection of $\Pi$ and $L$ is $q$. 
On the other hand, since $\Pi \cap K$ is centrally symmetric, $W$ must pass through $o$. It follows that $W=L(o,q)=\Lambda$.
The case where $L$ and $\Lambda$ are parallel can be considered analogously. 

In the case where $o=p$, every line $V\subset \Omega$ passing through $o$ is an axis of symmetry of $K$. By Theorem \ref{aux}, we conclude that $K$ is a body of revolution with an axis perpendicular to $\Omega$ and passing through $o$.\end{proof}

\noindent \textit{Proof of Theorem \ref{sisi}.} If $o=p$, by the second part of Lemma \ref{liz}, $K$ is a body of revolution and we are done. On the other hand, if $o\not= p$, by the first part of Lemma \ref{liz}, the line $\Lambda$ is an axis of symmetry of $K$. Furthermore, since $K$ has centre at $o$, every axis of symmetry of $K$ must pass through $o$, and it follows that $\Lambda$ is the unique axis of symmetry of $K$ containing $p$. We may have that $\Lambda$ intersects $L$ or that they are parallel.

We take a system of coordinates $(x_1,x_2,x_3)$ in $\Rt$ such that $p$ is the origin, 
$\Omega$ is given by the equation $x_3=0$ and $\Lambda$ corresponds to the axis of the first coordinate $x_1$. For each $\theta \in (-\pi/2,\pi/2]$, we denote by $L(\theta)$ the line making an angle $\theta$ with the positive axis $x_1$, by $\Omega(\theta,\phi)$ the plane containing $L(\theta)$ and making a positive angle  $\phi$ with the plane $x_3=0$, and by $K(\theta,\phi )$ the section $\Omega (\theta,\phi) \cap K$. We denote by $M$ the line obtained by reflecting $L$ with respect to $\Lambda$, {\it i. e.}, $M=R_{\Lambda}(L)$,  which is also contained in $\Omega$, (see Figure \ref{kari}).

Since $L$ and $\Lambda$ are not perpendicular, we have that $L\not=M$. Denote by  $q(\theta)$ and $m(\theta)$ the intersections of $L(\theta)$ with $L$ and with $M$, respectively.   For $\theta \not= \pi /2$, we have that
\begin{eqnarray}\label{oscar}
R_{\Lambda}(m(\theta)) \not= q(\theta).
\end{eqnarray} 
Let $D(\theta,\phi)$ be the line of symmetry of $K(\theta,\phi)$ passing through $q(\theta)$,  given by the hypothesis of Theorem \ref{sisi}. Also by the hypothesis, the plane $\Omega(-\theta,-\phi)= R_{\Lambda}(\Omega(\theta,\phi))$ has a line of symmetry $D((-\theta,-\phi)$ passing through $R_{\Lambda}(m(\theta)) \in L$.

\begin{figure}[H]
\centering
\includegraphics [width=6in] {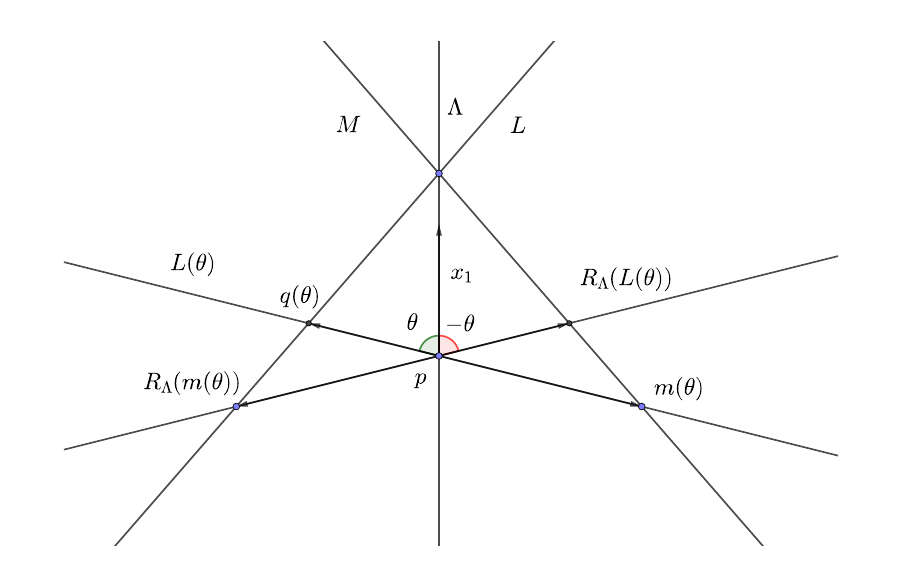} 
\caption{Case 1 of Theorem \ref{sisi}.}
\label{kari}
\end{figure}

Since $\Lambda$ is an axis of symmetry of $K$, we have that 
\[
R_{\Lambda}(K(-\theta, -\phi))=K(\theta,\phi).
\]
Thus, the section  $K(\theta,\phi)$ has two (possibly equal) lines of symmetry, namely $D(\theta,\phi)$ through $q(\theta)$ and 
$R_{\Lambda}(D(-\theta, -\phi))$ through $m(\theta)$. We have the following two possibilities:
\begin{itemize}
\item [(I)]   $D(\theta,\phi) =R_{\Lambda}(D(-\theta, -\phi))=L(\theta)$ or
\item [(II)]  $D(\theta,\phi) \not=R_{\Lambda}(D(-\theta, -\phi))$.
\end{itemize}
Suppose that there is a fixed $\theta_0 \in (-\pi/2,\pi/2], \theta_0\not=0$, such that (I) holds for $\theta_0$ and all $\phi \in (-\pi/2,\pi/2]$. Then $ L(\theta_0)$ must be  an axis of symmetry of $K$. But $L(\theta)$ contains the point  $p$ and, as we observed above, $\Lambda$ is the unique axis of symmetry of $K$ containing $p$.  
Therefore this situation is impossible.
In fact, we can weaken the hypothesis that (I) holds for a single $\theta_0$ and every $\phi$, as the next Lemma shows.

\begin{lemma}\label{kenny}
It is impossible for condition \textrm{(I)} to hold for a fixed  $\theta_0 \in (-\pi/2,\pi/2], \theta_0\not=0$,  and $\phi\in [\phi_1,\phi_2]$, $-\pi/2<\phi_1<\phi_2<\pi/2$. 
\end{lemma}

\begin{proof}
 Assume that condition (I) holds, {\it i.e} the line $L(\theta_0)$ is a line of symmetry of $K(\theta_0,\phi)$ for $\phi \in [\phi_1,\phi_2]$. Observe that the line $L(\theta_0)$ is not an axis of symmetry of $K$ (because $\Lambda$ is the only axis of symmetry of $K$ passing through $p$).  
Since $\Lambda$ is an axis of symmetry of $K$, the section $\{x_1=0\}\cap K$ is centrally symmetric with centre at $p$. In addition, by Lemma \ref{marina} we have that $t(p)= \partial K \cap \{x_1=0\}$. On the other hand, since $t(p)$ is a curve and, 
since $L(\theta_0)$ is a line of symmetry for $K(\theta_0,\phi)$, $\phi \in [\phi_1,\phi_2]$, 
there exists a plane $\Sigma_0$ perpendicular to $L(\theta_0)$ and containing $p$, such that 
for every $\phi\in [\phi_1,\phi_2]$ the chords $(\Omega(\theta_0,\phi)\cap \Sigma_0)\cap K$ belong to $C(p)$, \textit{i.e.}, there exist two arcs $\rho_0, \tau_0$ of $t(p)$ such that $\rho_0, \tau_0\subset \Sigma$ and $\tau_0=-\rho_0$ (See Figure \ref{mami}).  
Since $\theta_0\not=0$,  it follows that $\Sigma_0 \not= \{x_1=0\}$. Thus the arcs $\rho_0$ and $\tau_0$ would not be contained in $\{x_1=0\}$ contradicting Lemma \ref{marina}.

\begin{figure}[H]
\centering
\includegraphics [width=6in] {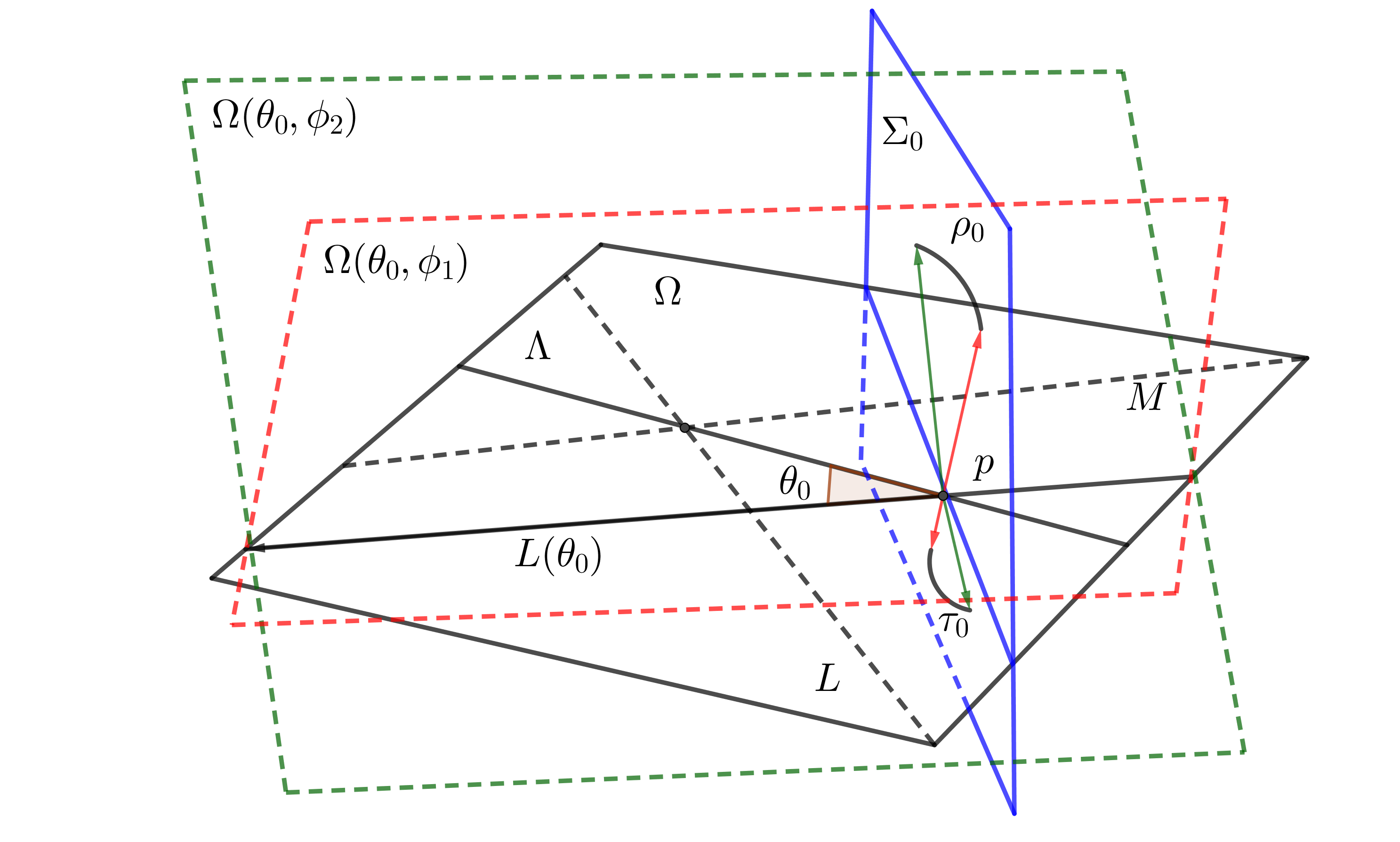} 
\caption{What if (I) holds for some $\phi$ and (II) for other $\phi$?}
\label{mami}
\end{figure} 
\end{proof}

Now we suppose that the case (II) is satisfied for all $\theta  \in (-\pi/2,\pi/2)$ and almost every $\phi$ in $(-\pi/2,\pi/2)$,  $\theta, \phi \not= 0$, 
\textit{i.e.}, the section  $K(\theta,\phi)$ has two different lines of symmetry, namely, $D(\theta,\phi)$ through $q(\theta)$ and $R_{\Lambda}(D(-\theta, -\phi))$ through $m(\theta)$.
We denote by  $E_m(\theta,\phi)$ the line $R_{\Lambda}(D(-\theta, -\phi))$, and by $z(\theta,\phi)$ the point of intersection of the lines of symmetry $D(\theta,\phi)$ and $E_m(\theta,\phi)$. For $\theta$ fixed, we consider the following two functions: $f_{\theta}:[0,\pi] \rightarrow [0,a]$, $a\in \mathbb{R^+}$, is defined as the distance from the point $z(\theta,\phi)$ 
to the line $L(\theta)$, and $g_{\theta}:[0,\pi] \rightarrow (0,\pi]$ is defined as the angle between the two axes of symmetry. 
 It follows directly from Lemma \ref{chinita} that, for $\theta$ fixed, \textit{the functions $f_{\theta}$ and $g_{\theta}$ are continuous as functions of $\phi$}. 
By the compactness of $[0,\pi]$ there exist $\alpha, \beta, \gamma,\delta \in \mathbb{R}$ such that $f_{\theta}(\alpha)\leq f_{\theta}(\phi)\leq f_{\theta}(\beta)$ and $g_{\theta}(\gamma)\leq g_{\theta}(\phi)\leq g_{\theta}(\delta)$. 

\begin{lemma}\label{simone}
For each $\theta \in[0,\pi]$ there exists $\phi_{\theta}\in [0,\pi]$ such that $K(\theta,\phi_{\theta})$ has $L(\theta)$ as line of symmetry.  
\end{lemma}
\begin{proof}
We will show that $f_{\theta}(\alpha)=0$. Denote by $u$  the unit vector $(\cos \theta,\sin \theta,0)$, which is the direction vector of the line $L(\theta)$. Let $\pi: \Rt \rightarrow u^\perp$ be the orthogonal projection corresponding to $u$. We define a map $\xi: \mathbb{S}^2\cap u^{\perp} \rightarrow \mathbb{R}^2$ as follows: for $\phi \in [0,\pi]$ we define $v=(\cos \phi,\sin \phi)\in S_u:=$$\mathbb{S}^2\cap u^{\perp}$ and $\xi(v)=\pi(z(\theta,\phi))$. This is a point on the line of intersection of the planes  $u^\perp$ and $\Omega(\theta,\phi)$. Naturally, we have $\xi (-v)=\xi(v)$ , where $-v=(\cos( \phi +\pi), \sin (\phi +\pi)  )$, because $-v$ and $v$ define the same plane $\Omega(\theta,\phi)$. Since $\xi$ is a continuous function (this follows from the continuity of $f_\theta$),   there exist a $\phi_0$ such that for $v_0=(\cos \phi_0,\sin \phi_0)$ we have $\xi(v_0)=0$, otherwise, the standard vector bundle of 1-dimensional sub-spaces of $\mathbb{R}^2$, $\gamma^1: E\rightarrow \mathbb{R}P^n$ would have a non-zero section, which it would be a contradiction (see Proposition 4 and Example 3 in Section 4 of \cite{milnor}). Consequently, $f_\theta(\phi_0)=0$, and $\phi_\theta=\phi_0$ is the number the we were looking for.
\end{proof}

\begin{lemma}\label{caminante}
For all $\theta \in [0,\pi]$ the equality $\phi_{\theta}=\frac{\pi}{2}$ holds.
\end{lemma}
\begin{proof}
Contrary to the statement of the Lemma, 
suppose that there exists $\theta_0\in[0,\pi]$ such that $\phi_{\theta_0}\not= \frac{\pi}{2}$. Then the inequality
\begin{eqnarray}\label{reik}
[\Omega(\theta_0,\phi_{\theta_0})\cap \{x_1=0\}]\not=[ \{x_1=0\}\cap\{x_2=0\}].
\end{eqnarray}
holds. We denote by $A,B$ the extreme points of the chord $[\Omega(\theta_0,\phi_{\theta_0})\cap \{x_1=0\}]\cap \partial K$ and by $A',B'$ the images of $A,B$ under the reflection in $\Omega(\theta_0,\phi_{\theta_0})$ with respect to the line $L(\theta)$. Notice, on the one hand, the segment $AB$ has midpoint $p$, since $\Lambda$ is an axis of symmetry of $K$, and on the other hand, $A'B'$ has midpoint $p$ by the definition of $\phi_{\theta_0}$. By equation \eqref{reik}, the chord $A'B'$ is not contained in the plane $x_1=0$. However, this contradicts Lemma \ref{marina} since $A'B'$ has midpoint at $p$.
\end{proof}

\begin{lemma}\label{cometa}
For all $\theta\in[0,\pi]$, the section $K(\theta,\phi_{\theta})$ is a circle.
\end{lemma}
\begin{proof}
Let $\{\phi_n\}\subset [0,\pi]$ be a sequence such that $\phi_n \rightarrow \phi_{\theta}$ when $n\rightarrow \infty$. By Lemma \ref{simone}, $f_\theta(\phi_\theta)=0$, {\it i.e.}, as $n \rightarrow \infty$ the point $z(\theta,\phi_n)$ converges to a point $z_0:=z(\theta,\phi_\theta)$ on the line $L(\theta)$. In this situation, we have three possible cases for the  sequence $g_\theta(\phi_n)$: either the angle between the two lines of symmetry converges to $\pi$ (if the point $z_0$ lies between $q(\theta)$ and $m(\theta)$, or it converges to 0 (if the point $z_0$ lies outside of the segment joining $q(\theta)$ and $m(\theta)$), or the angle $g_\theta(\phi_n)$ remains constant as $n \rightarrow \infty$, as is the case if the the point of intersection of the lines lies on an arc of a circle containing both $q(\theta)$ and $m(\theta)$, in which case $z_0$ is equal to $q(\theta)$ or $m(\theta)$.

First we consider the case where $g_\theta(\phi_n)\rightarrow \pi$ when $n\rightarrow \infty$. 
If the lines $D(\theta,\phi_n)$, $E(\theta,\phi_n)$ determine an $m(n)$-star, for some integer $m(n)$ depending of $n$, and for an infinite set of indices $n$, the assumption that $g_\theta(\phi_n)\rightarrow \pi$ implies that  $m(n)\rightarrow \infty$, (since $g_\theta(\phi_n)\rightarrow \pi$ then $[\pi-g_\theta(\phi_n)]\rightarrow 0$ and the number of lines of symmetries of $K(\theta,\phi_n)$ increases with $n$ and, consequently, $m(n)\rightarrow \infty$). Since $K(\theta,\phi_n)\rightarrow K(\theta,\phi_{\theta})$, by Proposition \ref{escandalo} it follows that $K(\theta,\phi_{\theta})$ is a circle. On the other hand, 
if the lines $D(\theta,\phi_n)$, $E(\theta,\phi_n)$ determine an $m(n)$-star, for some integer $m(n)$ depending of $n$ and for a finite set of indices $I=\{n_1,...,n_k\}$,   since $g_\theta(\phi_n)\rightarrow \pi$, then for $n\in \mathbb{N}\backslash I$ the angle $g_{\theta}(\phi_n)$ is irrational thus $K(\theta,\phi_{n})$ is a circle. By virtue of the fact that $K(\theta,\phi_n)\rightarrow K(\theta,\phi_{\theta})$,  it follows from  Proposition \ref{escandalo} that $K(\theta,\phi_{\theta})$ is a circle.
The argument in the case where $g_\theta(\phi_n)\rightarrow 0$ when $n\rightarrow \infty$ is similar.

Next, we will show that it is impossible to have $g_{\theta}(\phi_n)=k$ for some constant number $k$ and for all $n$. Assume, to the contrary, that this is the case. 
We will prove that either $z(\theta,\phi) \rightarrow q(\theta)$ or $z(\theta,\phi) \rightarrow m(\theta)$. Since $f_{\theta}(\phi_n)\rightarrow 0$, there exists $z_0\in L(\theta)$ such that $z(\theta,\phi) \rightarrow z_0$. Let $B_q, B_m$ be two balls with ratio $\epsilon$ and with centres at $q(\theta)$ and $m(\theta)$, respectively. Suppose that $w\in L(\theta)\backslash \{B_q \cup B_m\}$. Let $\{w_i\}\subset \Rt \backslash \{B_q \cup B_m \cup L(\theta)\}$ be a sequence such that $w_i \rightarrow w$ when $i\rightarrow \infty$. Since $\{w_i\}\subset \Rt \backslash \{B_q \cup B_m \cup L(\theta)\}$, the lines $L(q(\theta),w_i)$ and $L(m(\theta),w_i)$ are well defined. Given that $w_i \rightarrow w$, the angle $\alpha_i$ determined by  $L(q(\theta),w_i)$ and $L(m(\theta),w_i)$ tends either to $\pi$ (if $w$ is in the line segment determined by $q(\theta) $ and $m(\theta)$) or 0 (if $w$ is in the complement of line segment determined by $q(\theta)$ and $m(\theta)$). Consequently, $z_0= q(\theta)$ or $z_0= m(\theta)$, otherwise we would contradict the condition $g_{\theta}(\phi_n)=k$ for some constant number $k$ and for all $n$. 

In the case where $L\cap K=\emptyset$, we have that  $q(\theta),m(\theta)\in \Rt\backslash K$. Hence, since $z(\theta,\phi) \rightarrow q(\theta)$ or $z(\theta,\phi) \rightarrow m(\theta)$, there exist an integer $N$ such that $z(\theta,\phi) \in \Rt\backslash K$ for all $n>N$. On the other hand, since $z(\theta,\phi)$ is the intersection point of two lines of symmetry of $K(\theta,\phi_n)$ necessarily such point belongs to  $\inte K(\theta,\phi_n)$. This contradiction shows that the case where $g_{\theta}(\phi_n)$ is a constant sequence is impossible.

Alternatively, assume that  $L\cap K\not=\emptyset$. 
Recall that we denote by $C(x)$ the family of chords of $K$ whose midpoint is $x$, and by $t(x)$ the endpoints of these chords. We claim that $t(p(\theta))$ and $t(m(\theta))$ are curves with center at $p(\theta)$ and $m(\theta)$ respectively. For each $\theta, \phi\in [0,\pi]$ there exist two different lines of symmetry  $D(\theta,\phi)$ and $E(\theta,\phi)$ of $K(\theta,\phi)$, passing through $p(\theta)$ and $m(\theta)$, respectively. By symmetry, the chord $A(\theta,\phi)$ on $\Omega(\theta,\phi)$ passing through $p(\theta)$ and perpendicular to $D(\theta,\phi)$ is in the family  $C(p(\theta))$, and similarly, the chord $B(\theta,\phi)$ passing through $m(\theta)$ and perpendicular to $E(\theta,\phi)$ is in the family $C(m(\theta))$,
(see Figure \ref{ultima}). Hence, the endpoints of $A(\theta,\phi)$ belong to $t(p(\theta))$ and the endpoints of $B(\theta,\phi)$ belong to $t(m(\theta))$. Varying $\theta$ and $\phi$, we conclude that $t(p(\theta))$ and $t(m(\theta))$ are curves with centers at $p(\theta)$ and $m(\theta)$, respectively.

On the other hand, since $A(\theta,\phi) \perp D(\theta,\phi)\textrm{  } \textrm{ and } \textrm{   }B(\theta,\phi) \perp E(\theta,\phi)$, the angle between the lines generated by $A(\theta,\phi)$ and $B(\theta,\phi)$ is equal to $\pi-g(\theta,\phi)$. Notice that, by Lemma \ref{caminante}, $A(\theta,\pi/2)$ and $B(\theta,\pi/2)$ are orthogonal to $\Omega$, \textit{i.e.}, the chords $A(\theta,\pi/2)$ and $B(\theta,\pi/2)$ are parallel. By virtue of the fact that $g(\theta,0)>0$ (since the lines $D(\theta,0)$ $E(\theta,0)$ pass through $o$ and $p(\theta)\not= m(\theta)$), we have that  $g(\theta,\pi/2)=\pi$. Indeed, $A(\theta,\pi/2)$ and $B(\theta,\pi/2)$ are parallel, which means that the angle between the lines generated by $A(\theta,\pi/2)$ and $B(\theta,\pi/2)$ is equal to $0$ and, on the other hand, it is equal to $\pi-g(\theta,\pi/2)$. Consequently, $g(\theta,\pi/2)=\pi$,  and since $g(\theta,\phi)$ is continuous,  it cannot be constant in a neighborhood of $\pi/2$. Hence, also in this case we conclude that $g_{\theta}(\phi_n)$ cannot be  a constant sequence as $n \rightarrow \infty$,  {\it i.e.}, $\phi_n \rightarrow \pi/2$.
\end{proof}

\begin{figure}
\centering
\includegraphics [width=5.5in] {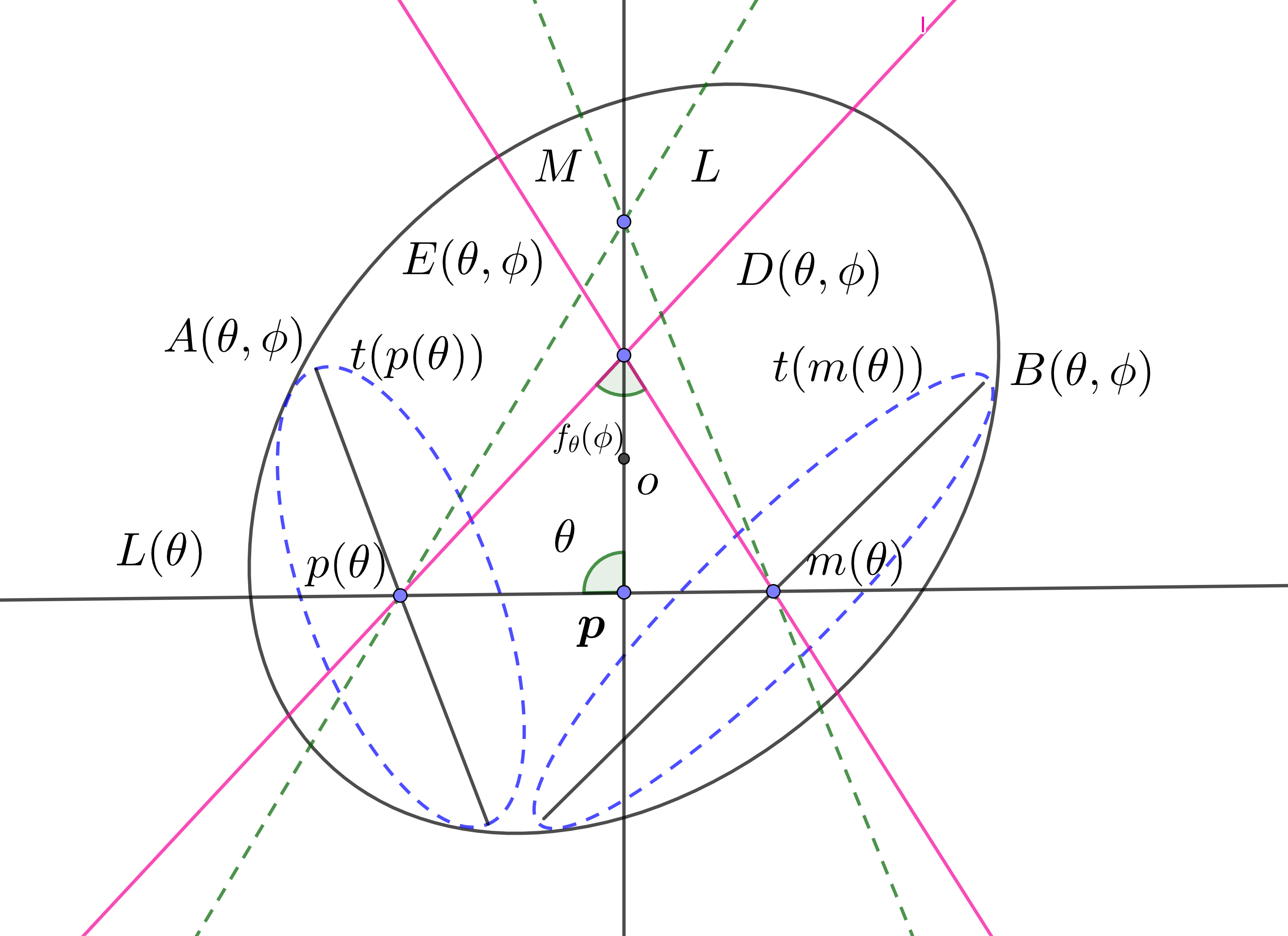} 
\caption{Case: $L\cap K\not=\emptyset$.}
\label{ultima}
\end{figure}

\begin{lemma}\label{astrud}
The section $K_0:=\Omega \cap K$ is a circle.
\end{lemma}
\begin{proof}
In order to prove that $K_0$ is a circle we are going to show that all the lines contained in $\Omega$, passing through the centre $o$ of $K_0$, are lines of symmetry of $K_0$. Let $W\subset \Omega$ be a line passing through the centre $o$ and let $\theta\in[0,\pi]$ such that $q(\theta)=W\cap L$, \textit{i.e.}, $L(\theta)$ meet $L$ at $W\cap L$. Let $\{\phi_n\}\subset [0,\pi]$ be a sequence such that $\phi_n \rightarrow 0$ when $n \rightarrow \infty$. It clear that $K(\theta, \phi_n) \rightarrow K_0$ when $n \rightarrow \infty$. Each section $K(\theta, \phi_n)$ has two lines of symmetry $D(\theta, \phi_n)$ and $E(\theta, \phi_n)$ passing through $q(\theta)$ and $m(\theta)$, respectively. 
By  Lemma \ref{chinita}, $K_0$ has a line of symmetry $T$ passing through $q(\theta)$ and, since $K_0$ has $o$ as a centre, $T$ is also passing through $o$. Thus $T=W$. Hence $K_0$ is a circle. 
\end{proof}

\begin{lemma}\label{sph}
$K$ is a sphere with centre at $o$.
\end{lemma}
\begin{proof}
By Lemma \ref{astrud}, $K_0$ is a circle. We suppose that the radius of $K_0$ is equal to 1. We are going to prove that $K$ is a sphere of ratio 1 with centre at $o$. Let $x \in \partial K$. Let $\theta\in [0,\pi]$ such that $x\in K(\theta,\phi_{\theta})$ (by Lemmas \ref{caminante} and  \ref{cometa} such $\theta$ exists). Let $r$ be a real number such that $o=(r,0,0)$. By Lemma \ref{cometa}, the section $K(\theta,\phi_{\theta})$ is a circle of radius
\[
\frac{|L(\theta)\cap K|}{2}=\sqrt{1-r^2 \sin^2 \theta}.
\]
Hence 
\[
||x-o||=a^2+\big{(}\sqrt{1-r^2 \sin^2 \theta} \big{)}^2=1,
\]
where $a=r\sin\theta$ is the distance from $o$ to $L(\theta)$. 
\end{proof}

In Theorem \ref{sisi}, the hypothesis asks that the line $L$ does not pass through the center of symmetry $o$ of $K$. The next result considers the case where  $o \in L$. In this case, we obtain the conclusion that $K$ is a body of revolution. The proof is very similar to the proof of Theorem \ref{tania}.


{\bf Proof of Theorem \ref{seven}}

\begin{proof}
In order to prove that $K$ is a body of revolution with axis $L$ we are going to show that every plane containing $L$ is a plane of symmetry of $K$. Let $\Pi$ be a plane, $L\subset \Pi$. Let $\Gamma$ be the plane perpendicular to $\Pi$, $L\subset \Gamma$ and we denote by $\Gamma^1$, $\Gamma^2$ the half-spaces defined by $\Gamma$. First we suppose that $p\notin \Gamma$, say $p\in \Gamma^1$. First, we will show that $\Gamma^2\cap K$ is a symmetric set with respect to $\Pi$. Let $W$ be a line perpendicular to $\Pi$, $p\in W$. Let $x\in \Gamma^2 \cap K $. We denote by $\Delta$ the plane generated by $x$ and $W$. We claim that
\[
(\Delta \cap \Gamma) \cap (L \cap \inte K)\not= \emptyset. 
\]
If $(\Delta \cap \Gamma) \cap (L \cap \inte  K)= \emptyset 
$, it would have that $\Delta \cap K\subset \Gamma^1$ but this would contradict that $x\in \Delta \cap K$ and $x\in \Gamma^2$. Let $\Sigma$ be the plane perpendicular to $L$ passing through $y:=(\Delta \cap \Gamma) \cap (L \cap \inte K)$. By  Lemma \ref{marina}, which we can use since $L$ is an axis of symmetry, the relation
\begin{eqnarray}\label{cafe}
t(y)=\Sigma \cap \partial K
\end{eqnarray}
holds. 
By hypothesis, there exists a line of symmetry $H$ of $\Delta \cap K$ passing through $y$. We claim that 
$H=\Pi\cap \Delta$.
Suppose that $H\not=\Pi\cap\Delta$. We denote by $\phi:\Delta \rightarrow \Delta$ the reflection with respect to the line $H$. Since $H\not=\Pi\cap\Delta$ it follows that, on the one hand, $\phi(M)\not=M$, where $M:=\Delta\cap \Gamma$, and, on the other hand, the line segment $\phi (M)\cap K$ has $y$ as mid-point. 
But, since  $\phi(M)\not=M$, $\phi (M)\cap K$ is not contained in $\Sigma \cap K$, which contradicts \eqref{cafe}.
\end{proof}

{\bf Acknowledgement:}
 This work was partially done during a sabbatical year of the fourth author at University College London (UCL). The fourth author thanks  UCL for their hospitality and support.  We thank Dmitry Ryabogin for many fruitful discussions about these results. We also thank the referee for many valuable suggestions to improve the clarity of the paper.

The first author is supported in part by the Simons Foundation gift  711907. The fourth author is supported by the National Council of Sciences and Technology of Mexico (CONACyT) Grant I0110/62/10 and SNI 21120.

\end{document}